\documentclass[11pt,reqno]{article}
\usepackage[colorlinks]{hyperref}
\AtBeginDocument{
	\hypersetup{
		colorlinks=true,
 		linkcolor=RoyalBlue,
  		citecolor=RedOrange,
	 }
}
\usepackage{amsmath,amsthm,amssymb}
\usepackage[usenames,dvipsnames]{xcolor}
\usepackage{pgf,tikz}
\usepackage{tkz-graph}
\usepackage[top = 1in, bottom = 1in, left = .75in, right =.75in]{geometry}
\usepackage{verbatim}
\usepackage{graphicx}
\usepackage{enumitem}
\usepackage{lineno}

\newcommand{\var}{{\rm var}}
\newcommand{\cov}{{\rm cov}}
\newcommand{\pr}[1]{\mathbb{P}\left( #1\right)}
\newcommand{\mean}[1]{\mathbb{E}\left[ #1\right]}
\newcommand{\sta}[1]{\mathcal{T}^*_{#1}}

 \newcommand{\bbP}[1]{\mathbb{P}\left( #1\right)}
 \newcommand{\bbE}[1]{\mathbb{E}\left[#1\right]}
 
\newcommand{\e}[1]{\exp \left( #1\right)}
\newcommand{\bs}{\boldsymbol}
\newcommand{\lp}{\left(}
\newcommand{\rp}{\right)}
\newcommand{\iab}{I_{a,b}^i}
\newcommand{\ibc}{I_{b,c}^i}
\newcommand{\icd}{I_{c,d}^i}
\newcommand{\ione}{I_{1,2}^{i}}
\newcommand{\um}[1]{\"{#1}}
\newcommand{\pag}[1]{\mathcal{G}_{#1}}
\def\T{{\mathcal T}}
\newcommand{\bni}{B_{n,i}}

\newcommand{\length}[1]{G_{#1}}

\def\ov{\overline}

\def\a{\alpha} 

\def\eps{\varepsilon}

\newtheorem{theorem}{Theorem}
\newtheorem{lemma}[theorem]{Lemma}
\newtheorem{corollary}[theorem]{Corollary}
\newtheorem{proposition}[theorem]{Proposition}

\theoremstyle{remark}

\newtheorem*{remark}{\bf Remark}

\title{On a  memory game and  preferential attachment graphs}
\date{}
\author{H\"{u}seyin Acan\\ School of Mathematical Sciences\\ Monash University\\ Melbourne, VIC 3800\\ Australia\\ 
\texttt{huseyin.acan@monash.edu}
\and
Pawe{\l} Hitczenko
\thanks{This author was partially supported by a 
Simons Foundation grant \#208766. His work was carried out during a visit at Monash
University in the first half of 2014. 
He would like to thank the members of the School of Mathematical
Sciences and Nick Wormald in particular for  hospitality  and 
support. He would also like to thank Mark Wilson for suggesting the problem and exchanging some ideas about  it. Both authors would like to thank Andrew Barbour for a helpful conversation on the topic related to this paper.}\\
Department of Mathematics\\ Drexel University\\ Philadelphia, PA 19104 \\ USA\\ \texttt{phitczen@math.drexel.edu}
}

\begin{document}

\maketitle
\begin{abstract} 
In a recent paper Velleman and Warrington  analyzed the expected  values of some of the parameters in a  memory game, namely, the length of the game, the waiting time for the first match, and the number of lucky moves.  In this paper we continue this direction of investigation  and obtain the limiting distributions of those parameters. More specifically, we prove that when suitably normalized, these quantities converge in distribution to a normal, Rayleigh, and Poisson random variable, respectively. 

We also make a connection between the memory game and one of the models of preferential attachment graphs.  In particular, as a by--product of our methods we obtain simpler proofs (although without rate of convergence) of some of the results of Pek\"oz, R\"ollin, and Ross on the joint limiting distributions of the degrees of the first few vertices in preferential attachment graphs.

For proving that the length of the game is asymptotically normal, our main technical tool is a limit result for the joint  distribution of the number of balls in  a multi--type generalized P\'olya urn model.

\vspace{.5cm}

\end{abstract}

\section{Introduction}\label{s:intro}

Consider the following game of memory played with a deck of cards, which has been studied by Velleman and Warrington~\cite{VW}. There are $2n$ cards with $n$ different labels, two cards for each label. They are shuffled and put in a row, all facing down. 
A player flips over two cards at each turn and if the cards match, she removes the cards from the deck; otherwise she flips them over again. The goal is to finish the game using the smallest number of moves. 
Each initial configuration of the cards corresponds to a permutation
of the multiset $\{1,1,2,2,\dots,n,n\}$. Following
Velleman and Warrington we assume that the player has a perfect memory and the initial configuration is given by a permutation chosen uniformly at random. 

The player uses the following optimal strategy. She starts from the
leftmost card and flips over the cards from left to right. Before
starting round $t$, if she knows the places of two identical cards,
i.e.\ if she flipped over one of them at round $t-1$ and one of them
before round $t-1$, then she removes the pair at round~$t$. Otherwise,
she flips over the first card that has not been flipped over yet. In
this case, if the matching card has been flipped over previously,
then she flips it over once again and removes the pair from the deck;
if not, she flips over the next unflipped card. Note that each card is
flipped over at least once and at most twice. Thus the player finishes
in at least $n$ rounds and at most $2n$ rounds.

We refer the reader to~\cite{VW} for more details and references about the game. In
particular, the authors of~\cite{VW} studied the expected values of three parameters: the expected length of the game (i.e.\ the total
number of moves), the expected number of flips till the first match,
and the expected number of lucky moves,
where a lucky move refers to a 
move at which two (necessarily consecutive) cards of the same label are flipped for the first (and the last) time in the same round. 

In this work
we take the analysis one step further and identify
the limiting distributions of those three parameters as the size of
the deck, $n$, increases to infinity;  see 
Theorems~\ref{cor:G_n},~\ref{thm:firstmatch}, and~\ref{thm:lucky}
 below  for  precise statements.  Furthermore, we describe below a connection between  the memory game and a version of a model of the preferential attachment graphs introduced in~\cite{BA}, mathematically formalized in~\cite{BRST}, and subsequently studied in a number of papers; see  e.g.~\cite{BBCS,BR04,R} and references therein. 
 Thus, our methods may be used to study the preferential attachment models. In particular, our  analysis allows us to find the asymptotic joint degree distribution of the first $k$ vertices of a preferential attachment graph, a result that was  originally obtained (in a more precise version) in~\cite{PRR2}.
 
This connection between the memory game and the preferential attachment graphs is made through {\em chord diagrams}, 
i.e.\  pairings of $2n$ points. 
To describe it,  we say that two games are equivalent if by relabeling the cards in one of them (matching cards receive the same labels) we can obtain the other one. 
Thus, each game is equivalent to a unique game in which the second occurrence of card $i$ precedes the second occurrence of card $j$ for $i<j$. Such games (called  standard in~\cite{VW})  correspond to chord diagrams.

Various enumeration problems about chord diagrams have been studied widely; see for example \cite{AP,CM,DP,FN,Rio75}. Besides their combinatorial significance, chord diagrams appear in various fields in mathematics, especially in topology. 
For detailed information about chord diagrams and their topological and algebraic significance we refer the reader to Chmutov, Duzhin, and Mostovoy's book~\cite{CDMbook}, and for several other applications, to the paper by Andersen et al.~\cite{Andetal13} and the references therein.

As another application, Bollob\'{a}s et al.\ \cite{BRST} used linearized chord diagrams to generate a preferential attachment random graph introduced by Barab\'{a}si and Albert \cite{BA}. This 
graph is obtained from a discrete time random graph process. First we describe the process as suggested by Bollob\'as et al.\  and then give their alternative way of obtaining the same graph. 
The process starts with $\pag{1}$, the graph with one vertex $v_1$ and  a loop on $v_1$. For $n>1$, $\pag{n}$ is obtained from $\pag{n-1}$ by creating a new vertex $v_n$ and adding a random edge $e$ incident to $v_n$, where 
\[
\pr{e=v_iv_n}= \begin{cases}
D_{\pag{n-1}}(v_i)/(2n-1) & \text{ if }1\le i \le n-1, \\
1/(2n-1) & \text{ if } i=n,
\end{cases}
\]
and where $D_{\pag{}}(v)$ denotes the degree of the vertex $v$ in the graph $\pag{}$. (A loop contributes 2 to the degree.)

Bollob\'as et al.\   proposed the following method to generate $\pag{n}$. On a line, take $2n$ points and pair them randomly. Connect the points in each pair by an arc (chord) above the line. This is a random linearized chord diagram, call it $C_n$. For each arc, there is a left and a right endpoint.
Proceeding from left to right, identify all endpoints up to and containing the first right endpoint to form the vertex $v_1$. Then, starting with the next endpoint, identify all the endpoints up to the second right endpoint to form $v_2$, and so on. Then, for each arc, create an edge (possibly a loop) between the vertices corresponding to the left and right endpoints of the arc; see Figure~\ref{fig:LCD and graph}. The graph $\pag{n+1}$ can be obtained from $\pag{n}$ as follows. Create one endpoint to the right of the last endpoint in $C_n$ and then create another endpoint in one of the $2n+1$ intervals, choosing the interval uniformly at random. Finally, join these two endpoints by an arc and modify the graph to obtain $\pag{n+1}$. 

Using this alternative description, Bollob\'as et al.~\cite{BRST} proved that this graph has power law degree distribution,
which had been previously observed by Barab\'{a}si and Albert \cite{BA} using experimental methods. (Power law degree distribution means that the number of vertices with degree $k$ is proportional to $nk^{-\gamma}$ as $n \to \infty$, for some $\gamma>0$. Bollob\'as et al.\ showed that $\gamma=3$ in particular.) Later Bollob\'as and Riordan~\cite{BR04} studied the diameter of $\pag{n}$ making use of linearized chord diagrams again.

\begin{figure}[t]
\centering 
\begin{tabular} {c c}
\begin{tikzpicture}[scale=1.2, line cap=round,line join=round,>=triangle 45,x=1.0cm,y=1.0cm]
\draw[gray] (-.5,0)--(6.5,0);
\foreach \x in {1,...,4}
{
\path (0.66*\x-.66,0)++ (0,0) coordinate (n\x);
\fill[red] (n\x) circle[radius=1.5pt] ++ (0,-.3) node {};
}
\foreach \x in {5}
{
\path (0.66*\x-.66,0)++ (0,0) coordinate (n\x);
\fill[blue] (n\x) circle[radius=1.5pt] ++ (0,-.3) node {};
}
\foreach \x in {6}
{
\path (0.66*\x-.66,0)++ (0,0) coordinate (n\x);
\fill[ForestGreen] (n\x) circle[radius=1.5pt] ++ (0,-.3) node {};
}
\foreach \x in {7,8,9}
{
\path (0.66*\x-.66,0)++ (0,0) coordinate (n\x);
\fill[Fuchsia] (n\x) circle[radius=1.5pt] ++ (0,-.3) node {};
}
\foreach \x in {10}
{
\path (0.66*\x-.66,0)++ (0,0) coordinate (n\x);
\fill[orange] (n\x) circle[radius=1.5pt] ++ (0,-.3) node {};
}
\foreach \from/\to in {n1/n4, n2/n5, n3/n6, n7/n9, n8/n10}
\path (\from) edge[bend left=90] (\to);
\end{tikzpicture}

\qquad

\begin{tikzpicture}
 [scale=.7,auto=left,every node/.style={ inner sep=2pt}, every loop/.style={min distance=20mm}]
\coordinate (n1) at (0,0) {};
\coordinate (n2) at (1,1) {};
\coordinate (n3)  at (2,0) {};
\coordinate (n4)   at (3,1) {};
\coordinate (n5) at (3,0) {};

\fill[blue] (n1) circle[radius=3pt] node[below] {$v_2$};
\fill[red] (n2) circle[radius=3pt] node[above] {$v_1$};
\fill[ForestGreen] (n3) circle[radius=3pt] node[below] {$v_3$};
\fill[Fuchsia] (n4) circle[radius=3pt] node[above] {$v_4$};
\fill[orange] (n5) circle[radius=3pt] node[below] {$v_5$};

\foreach \from/\to in {n1/n2, n2/n3, n4/n5}
  \draw (\from) -- (\to);
  \path (n2) edge[loop] (n2);
  \path (n4) edge[loop] (n4);
\end{tikzpicture}
\end{tabular}
\caption{A linearized chord diagram and the corresponding graph.}
\label{fig:LCD and graph}
\end{figure}
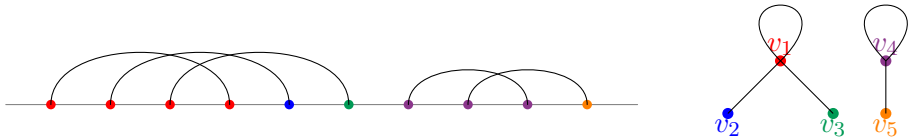

Recently, Pek\"{o}z et al.~\cite{PRR} studied the degree of a fixed vertex in $\pag{n}$. In particular, they found
an asymptotic distribution for $D_{\pag{n}}(v_i)$, along with a rate of convergence, as $n\to \infty$. In a more recent work~\cite{PRR2}, they extended their result to the joint distribution of the first $k$ vertices. One consequence of our approach is 
a simpler proof for the asymptotic joint degree distribution of the first $k$ vertices for any fixed $k$ using chord diagrams and the above description of $\pag{n}$.

In the next section we will introduce our notation and present our main results. In the subsequent sections, we will give the proofs of these results.

\section{Main results}
\label{sec:main}

The \emph{size} of a memory game is the number of pairs of cards a
player starts with. Suppose that we have an infinite number of cards
labeled $1,2,\dots$ in pairs, and each pair has a blue copy and a
red copy. The colors are only on one side, and when a card is face
down  one cannot see its color. A game of size $n$ is simply a permutation of the first $n$ pairs of cards. 
Since the colors are not essential for the game, we may assume that a
blue card always precedes the red card with the same label. Let $\T_n$
denote the set of such permutations with $n$ pairs. Thus the
cardinality of $\T_n$ is $(2n)!/2^n$. Also let $\sta{n}$ denote the
subset of $\T_n$ such that the red card with label $i$ comes before the
red card with label $i+1$ for every $i\in [n-1]$.
Following~\cite{VW}, we call $\sta{n}$ 
the set of \emph{standard deals} (denoted by $M_n$ in \cite{VW}).  It is easy to see that the cardinality of $\sta{n}$ is the product of odd positive integers 1 through $2n-1$.

A permutation $\sigma \in \T_n$ can be standardized by relabeling the pairs of cards so that the red cards appear in increasing order. Such a relabeling does not change the game's characteristics that we are interested in. Note that there is a unique standard deal corresponding to a permutation in $\T_n$ and conversely, there are $n!$ permutations equivalent to a given standard deal in $\sta{n}$. A (uniformly) random game corresponds to a (uniformly) random element of $\T_n$ or $\sta{n}$. We are free to choose either interpretation for our analyses.

Red cards partition a permutation in $\T_n$ into $n$ disjoint blocks. Each block consists of consecutive cards in the permutation, and for $1\le k\le n$, the $k$-th block ends with the $k$-th red card in the permutation.
The \emph{length of a block} is the number of cards it contains. We denote by $B_{n,i}$ the number of blocks of length $i$ in a random game with $n$ pairs of cards. We have
\[
\sum_{i\ge 1}B_{n,i}=n
\]
and $B_{n,i}=0$ for every $i>n+1$. 

We will use the symbol $(n)_k$ for the falling factorial. In particular, $(i+2)_3=i(i+1)(i+2)$ will appear frequently throughout the paper.
The following result gives the joint asymptotic distribution of the block counts and is crucial for our analysis. 
\begin{theorem}\label{thm:joint} 
As $n\to\infty$  
\begin{equation}\label{eqn:joint}\frac1{\sqrt
 n}\left(B_{n,i}-\frac{4n}{(i+2)_3}\right)_{i=1}^\infty\stackrel
d\to \big(W_i\big)_{i=1}^\infty,\end{equation}
where $(W_i)$ are jointly Gaussian with mean zero and covariance
matrix $\Sigma=[\sigma_{ij}]$ given by 
\begin{equation}\label{eqn:cov}
\sigma_{ij}=\left\{
\begin{array}{ll}\frac{16}{(i+2)_3(j+2)_3}-\frac{24}{(i+j+2)_4},&
\mbox{\ if\ }i\ne
j;\\ &\\ \frac4{(j+2)_3}+\frac{16}{(j+2)_3^2}-\frac{24}{(2j+2)_4},&\mbox{\
  if \ }i=j.\end{array}
\right.
\end{equation}
\end{theorem}

The following theorem gives us the total number of even length blocks. This is closely related to the \emph{length of the game} defined as the minimum number of moves required to finish the game. 

\begin{theorem}\label{thm:main}
Let $Y_n=\sum_{i\ge1}B_{n,2i}$. Then, as $n\to\infty$
\begin{equation}\label{joint_norm}\frac1{\sqrt{n}}\left(Y_n-(3-4\ln2)n\right)\stackrel d\to N(0,\sigma^2),\end{equation}
where $\sigma^2=(4\ln2)^2+8\ln2-13\sim0.2324$.
\end{theorem}

Denote by $\length{n}$ the length of a random game with $n$ pairs of cards. Using the previous theorem, we obtain the distribution of $\length{n}$, which is stated next.
\begin{corollary}\label{cor:G_n}
As $n \to \infty$, we have
\[
\frac{\length{n}-(3-2\ln2)n}{\sqrt n}\stackrel d\to N\left(0,\sigma^2/4\right),
\]
where $\sigma^2$ is as in Theorem~\ref{thm:main}.
\end{corollary}

\begin{proof}

This follows immediately from Theorem~\ref{thm:main} and the fact that the length of the game, $\length{n}$, satisfies 
\[
\length{n} = \frac32n+\frac12\sum_{i\ge1}B_{n,2i}- L_n
\]
as observed by Velleman and Warrington~\cite[Section~4]{VW}. Here $L_n$ denotes the number of lucky moves (defined in~\cite{VW} and also below) and by Theorem~\ref{thm:lucky} $L_n$ is bounded in probability.
\end{proof}

Denote by  $D_{n,i}$ the length of the $i$-th block in a random game. In the random preferential attachment graph described in Section~\ref{s:intro}, $D_{n,i}$ corresponds to the degree of the $i$-th vertex. The distribution of $D_{n,i}$ is given by Pek\um{o}z et al.\ \cite{PRR}. The distribution of $D_{n,1}$ is given in the following theorem. In Section~\ref{s:firstblock} we provide a short proof for this theorem.

\begin{theorem}\label{thm:firstmatch}
Let $X$ be a standard Weibull random variable with parameter 2, i.e.\ a
random variable whose probability density function is $2xe^{-x^2}$
if $x>0$ and 0 otherwise. Then as $n\to\infty$
\begin{equation}\label{eqn:1st_match} \frac{D_{n,1}}{2\sqrt n}\stackrel d\longrightarrow X,\end{equation}
where $\stackrel d\longrightarrow$ denotes the convergence in
distribution.
\end{theorem}

The next theorem is a generalization of the previous theorem. This result recovers, albeit without any error bound, the result of Pek\um{o}z et al.\ \cite[Theorem 1.1]{PRR2}.

\begin{theorem}\label{thm:firstkmatch}
As $n\to \infty$, for any fixed positive integer $k$,
\[
\lp \frac{D_{n,1}}{2\sqrt n},\dots, \frac{D_{n,k}}{2\sqrt n} \rp \stackrel d\longrightarrow (X_1,\dots,X_k),
\]
where the joint density function of $(X_1,\dots,X_k)$ is given by
\[
f(x_1,\dots,x_k)= 
\begin{cases}
2^k\, x_1(x_1+x_2)\cdots(x_1+\cdots+x_k) e^{-(x_1+\cdots+x_k)^2} & \text{if } x_1,\dots,x_k \ge 0\,,	\\
0 & \text{otherwise}\,. 
\end{cases}
\]
\end{theorem}

A \emph{lucky move} is a move at which we pick two matching cards by mere chance. This happens when two matching cards are next to each other and neither of these two cards has been flipped over before  this move. We denote by $L_n$ the number of lucky moves in a random game with $n$ pairs of cards.

\begin{theorem}\label{thm:lucky}
As $n\to\infty$, the number of lucky moves $L_n$ converges in distribution to a Poisson random variable with parameter $\ln 2$, that is,
\[
L_n\stackrel d\to Pois(\ln2).
\]
\end{theorem}

\section{Length of the first block}\label{s:firstblock}

Recall that the length of the first block corresponds to the degree of
the first vertex in the random preferential attachment graph
$\pag{n}$. When standard deals are used, 
 the first match occurs at the
position of the  red card with value 1.  For $\sigma\in\sta{n}$  Velleman and Warrington  denoted this position by $f(\sigma)$ and 
they defined a sequence  
\[
a(n,j):=\{\pi \in\sta{n}: f(\pi)=j\}.
\]
They noted  (\cite[Section~3]{VW}) that the numbers $a(n,j)$ satisfy the recurrence
\[a(n,j)=(2n-1-j)a(n-1,j)+(j-1)a(n-1,j-1)
\]
with the boundary conditions $a(1,2)=1$ and $a(n,j)=0$ if $j<2$ or
$j>n+1$. They used this relation to obtain the expression for the expected position
of the first match. 
Their argument quickly leads to the following statement.
\begin{proposition}\label{prop:A_recur} Let 
\[A_n(x)=\sum_{j=2}^{n+1}a(n,j)x^j,\] 
be the generating polynomial of the sequence $(a(n,j))$. Then, for $n\ge2$
\begin{equation}\label{rec_An}A_n(x)=(2n-1)A_{n-1}(x)+x(x-1)A'_{n-1}(x),\quad A_1(x)=x^2.\end{equation}
Furthermore, (see \cite[Theorem~5]{VW} for the first statement) as $n\to\infty$
\begin{equation}\label{expDn1}\bbE{D_{n,1}}=\frac{2^{2n}}{{2n\choose n}}\sim\sqrt{\pi n}\end{equation} 
and
\begin{equation}\label{varDn1}\var(D_{n,1})= (4-\pi)n+O(\sqrt n).\end{equation}
\end{proposition}
\proof Let $h(j)$ be any sequence and set
\[H_n:=\sum_{j=2}^{n+1}h(j)a(n,j).\]
Then, by 
\cite[Lemma~4]{VW} (where $h(j)=j$ is used   
but the same argument
works for any sequence $(h(j))$) we obtain
\[
H_n=(2n-1)H_{n-1}+\sum_{j=2}^{n}j(h(j+1)-h(j))a(n-1,j).
\]
If $h(j)=x^j$ then $H_n=A_n(x)$, $h(j+1)-h(j)=x^j(x-1)$ and we obtain
\eqref{rec_An}.

To prove \eqref{expDn1} and \eqref{varDn1} let $r\ge 1$. The Leibnitz formula and \eqref{rec_An} imply that 
\begin{eqnarray*}A^{(r)}_n(x)&=&(2n-1)A_{n-1}^{(r)}(x)+\Big(x(x-1)A'_{n-1}(x)\Big)^{(r)}\\&=&
(2n-1)A_{n-1}^{(r)}(x)+\sum_{j=0}^r{r\choose j}(x(x-1))^{(j)}A_{n-1}^{(r-j+1)}(x)
\\&=&(2n-1)A_{n-1}^{(r)}(x)+x(x-1)A_{n-1}^{(r+1)}(x)+r(2x-1)A_{n-1}^{(r)}(x)+r(r-1)A_{n-1}^{(r-1)}(x)
\\&=&(2n-1+r(2x-1))A_{n-1}^{(r)}(x)+r(r-1)A_{n-1}^{(r-1)}(x)+x(x-1)A_{n-1}^{(r+1)}(x).
\end{eqnarray*}
It follows that 
\[ 
A_n^{(r)}(1)=(2n-1+r)A_{n-1}^{(r)}(1)+r(r-1)A_{n-1}^{(r-1)}(1).
\]
When $r=0$ we obtain
\[A_n(1)=(2n-1)A_{n-1}(1)=\dots=(2n-1)!!\]
as was observed by Velleman and Warrington \cite[Lemma~3]{VW}.  
 This yields the following recurrence for the factorial moments
\begin{eqnarray*}\bbE{(D_{n,1})_r}&=&\frac{A_n^{(r)}(1)}{A_n(1)}=(2n-1+r)\frac{A_{n-1}^{(r)}(1)}{A_n(1)}+r(r-1)\frac{A_{n-1}^{(r-1)}(1)}{A_n(1)}\\&=&
\frac{2n-1+r}{2n-1}\frac{A_{n-1}^{(r)}(1)}{A_{n-1}(1)}+\frac{r(r-1)}{2n-1}\frac{A_{n-1}^{(r-1)}(1)}{A_{n-1}(1)}\\&
=&\frac{2n-1+r}{2n-1}\bbE{(D_{n-1,1})_r}+\frac{r(r-1)}{2n-1}\bbE{(D_{n-1,1})_{r-1}}
\end{eqnarray*}
with the initial condition specified by $D_{1,1}=2$.
Letting $r=1$ gives
\[
\bbE{D_{n,1}}=\frac{2n}{2n-1}\bbE{D_{n-1,1}} =
\frac{(2n)!!}{(2n-1)!!}=\frac{2^{2n}}{{2n\choose n}}\sim\sqrt{\pi n} 
\]
which proves \eqref{expDn1}.

When we set $r=2$ we get
\[\bbE{(D_{n,1})_2}
=\frac{2n+1}{2n-1}\bbE{(D_{n-1,1})_2}+\frac2{2n-1}\bbE{D_{n-1,1}}=
\frac{2n+1}{2n-1}\bbE{(D_{n-1,1})_2}+
\frac{2^{2n-1}}{(2n-1){2n-2\choose n-1}}.
\]
Letting $g_n:=\bbE{(D_{n,1})_{2}}/(2n+1)$ this yields
\[
g_n=g_{n-1}+\frac{2^{2n-1}}{(2n+1)(2n-1){2n-2\choose n-1}}
\]
with $g_1=2/3$.
This is solved by
\[
g_n=2-\frac{2^{2n+2}}{(n+1){2n+2\choose n+1}}		\sim2-\sqrt{\frac\pi{n+1}} .
\]
Hence
\[
\bbE{(D_{n,1})_2}=2(2n+1)-\frac{2n+1}{n+1}\frac{2^{2n+2}}{{2n+2\choose n+1}}=4n+O(\sqrt n)
\]
and thus
\[
\var(D_{n,1})=\bbE{(D_{n,1})_2}+\bbE{D_{n,1}}-(\bbE{D_{n,1}})^2=(4-\pi)n+O(\sqrt n),
\]
which proves \eqref{varDn1}.
\endproof

\begin{remark}
Presumably, this process could be continued to determine the
asymptotics of the factorial moments, and possibly identify the
limiting distribution of $D_{n,1}$.  For instance, for $r=3$, we can get
\[
\bbE{(D_{n,1})_3}=6\left(\frac{\sqrt\pi(n+2)n!}{\Gamma(n+1/2)}-4n-3\right)\sim6\sqrt\pi n^{3/2}.
\]
However, computations become progressively more complicated and hence we use a different approach to find an asymptotic distribution for $D_{n,1}$. 
\end{remark}

The following simple lemma will be used throughout the paper.
\begin{lemma}\label{lem:fallfact}
Let $n\to \infty$ and let $m=m(n)$ be a positive integer valued function such that  $m=o(n^{2/3})$. We have
\[
(n)_m = n^m \cdot e^{-{m\choose 2}/n} \left( 1+O(m^3/n^2)\right). 
\]
\end{lemma}

\begin{proof}
The desired equation follows from the Taylor expansion of the logarithms on the right--hand side of
\begin{equation}\label{eq:taylor}
(n)_m = n^m \cdot \prod_{j=1}^{m-1} (1-j/n)=n^m\e{\sum_{j=1}^{m-1}\ln (1-j/n)}.				\qedhere
\end{equation}
\end{proof}

\begin{proof}[Proof of Theorem~\ref{thm:firstmatch}]
For the proof, we use the space of all deals $\T_n$. Let $r_1$ denote the first red card in the random permutation. If the length of the first block is $t$, then the position of $r_1$ is $t$, in which case $t-1$ blue cards precede $r_1$, one of them being the blue counterpart of $r_1$. 
 Now we want to find the number of games such that the first red card appears at position $t$. To this end, we first choose $r_1$, which can be done in $n$ ways. Then we choose and permute the blue cards preceding $r_1$,  which can be done in ${n-1\choose t-2}(t-1)!$ ways. Note that we only choose $t-2$ blue cards since one of the $t-1$ is already known. Then we permute the remaining cards that will appear after $r_1$, and this can be done in $(2n-t)!2^{n-t+1}$ ways. Lastly, to find the probability $\bbP{D_{n,1}=t}$, we need to divide the product by $(2n)!2^{-n}$, the size of $\T_n$.
 Hence, for $2\le t\le n+1$, 
\begin{align}\label{probD_1=t}
\bbP{D_{n,1}=t} &= n{n-1\choose t-2}(t-1)! \frac{(2n-t)!}{2^{n-t+1}} \frac{2^n}{(2n)!}	 \notag	 \\
&= \frac{t-1}{2n-t+1} \, \frac{2^{t-1}\,(n)_{t-1}}{(2n)_{t-1}}.
\end{align}
Let $t=2\alpha\sqrt{n}$ for some positive $\alpha$. Using Lemma~\ref{lem:fallfact}, we have
\[
\frac{(n)_{t-1}}{(2n)_{t-1}}= \e{-t^2/4n+O(t/n)+O(t^3/n^2)},
\]
and consequently,
\[
\bbP{\frac{D_{n,1}}{2\sqrt{n}}=\alpha}= \lp 1+O\lp \frac{1}{\sqrt n}\rp \rp\frac1{2\sqrt n}\, 2\alpha e^{-\alpha^2}.
\]
Thus, $D_{n,1}/2\sqrt n$\, converges in distribution to a random variable $X$ with density function $f(x)$, where $f(x)=0$ for $x<0$, and 
$f(x)=2xe^{-x^2}$ for $x\ge 0$. In particular, we have
\[
F_n(z):=\bbP{\frac{D_{n,1}}{2\sqrt{n}}\le z} \to 1-e^{-z^2}
\]
for any $z\ge 0$.
\end{proof}

With a little more work, we can bound $|F_n(\alpha)-F(\alpha)|$, where $F$ is the cumulative distribution function of $X$.  Define $e_t= (t/(2n))\exp(-t^2/4n)$. Note that the sequence $\{e_t\}_{t\ge 1}$ is log-concave. Now fix $\alpha>0$ and let $N_\alpha=\lfloor2\alpha\sqrt n \rfloor$. 
By~\eqref{probD_1=t}, we have
\begin{align*}
F_n(\alpha) &=  \sum_{t=2}^{N_\a} \pr{D_{n,1}=t}  =  \sum_{t=1}^{N_\a-1} \frac{t}{2n-t} \cdot  \frac{2^t(n)_t}{(2n)_t}  =  \sum_{t=1}^{N_\a-1} \left( \frac{t}{2n}+O(t^2/n^2) \right) \cdot  \frac{2^t(n)_t}{(2n)_t}.
\end{align*}
Using Lemma~\ref{lem:fallfact},
\[
\frac{2^t(n)_t}{(2n)_t} = \e{-\frac{t(t-1)}{4n}}\left( 1+O\left(\frac{t^3}{n^2}\right)\right). 
\]
Combining the previous two equations,
\[
F_n(\alpha) = \sum_{t=1}^{N_\a-1} e_t \left( 1+O\left(t/n+t^3/n^2\right) \right) = \sum_{t=1}^{N_\a-1} e_t +O\big(1/\sqrt n\big).
\]
In fact, by a more careful analysis of \eqref{eq:taylor}, we can write
\[
\Big| F_n(\alpha)-\sum_{t=1}^{N_{\alpha}-1}e_t \Big| \le \frac{C(1+\alpha+\alpha^3)}{\sqrt{n}}
\]
for an absolute constant $C$ uniformly for all $\alpha$.
 Considering $\sum e_t$ as a Riemann sum, by the log-concavity of $\{e_t\}$, it is easy to see that, for some constant $C_2$,
\[
\Big| \sum_{t=1}^{N_{\alpha}-1}e_t - F(\alpha)\Big| \le C_2e^*, 
\]
where $e^*$ is the maximum of $\{e_t\}$, which is at most $1/\sqrt{2en}$. Using the last two equations, for all $\alpha>0$,
\[
|F_n(\alpha)-F(\alpha)| \le C_3(1+\alpha+\alpha^3)n^{-1/2}
\]
for an absolute constant $C_3$. On the other hand, for $\a \ge \sqrt{\log n}$,
\[
1-F(\a) = e^{-\a^2} =O(1/n)
\]
and
\[
|F_n(\a)-F(\a)| \le |1-F(\a)| +|F_n(\a)-1| \le O(1/n) +1-F_n\Big(\sqrt{\log n}\Big) = O\Big((\log n)^{3/2}n^{-1/2}\Big).
\]

\begin{proof}[Proof of Theorem~\ref{thm:firstkmatch}]
Let $t_1,\dots,t_k$ be positive integers and let $t=t_1+\cdots+t_k$. Similarly to~\eqref{probD_1=t}, we write
\begin{align*}
\bbP{D_{n,1}=t_1,\dots,D_{n,k}=t_k}&= {n \choose k}k! (t_1-1)(t_1+t_2-3)\cdots(t_1+\cdots+t_k-2k+1)\\
&\qquad \times {n-k\choose t-2k}(t-2k)! \frac{(2n-t)!}{2^{n-t+k}}\cdot \frac{2^n}{(2n)!}.
\end{align*}
Here $(t_1-1)$ is the number of possible positions for the blue partner of the first red card, $(t_1+t_2-3)$ is the number of positions for the blue partner of the second red card and so on. Simplifying this, we get
\[
\bbP{D_{n,1}=t_1,\dots,D_{n,k}=t_k}= \frac{2^{t-k}\,(n)_{t-k}}{(2n)_{m-k}}\cdot \frac{(t_1-1)\cdots(t_1+\cdots t_k-2k+1)}{(2n-m+k)_k}.
\]
Letting $t_i= 2\alpha_i\sqrt{n}$ for positive $\alpha_i$, the right--hand side becomes
\[
\lp 1+O\lp \frac{1}{\sqrt n}\rp \rp \lp e^{-(\alpha_1+\cdots+\alpha_k)^2}\rp \left(\frac{\alpha_1}{\sqrt n}\right)\,\cdots\, \left(\frac{\alpha_1+\cdots+\alpha_k}{\sqrt n}\right),
\]
from which the theorem follows.
\end{proof}

\section{Blocks of a given size}\label{sec:expectedblocksizes}

In this section we will estimate the expected value and the variance of $B_{n,i}$. We need these estimates in the next sections. 
We start with the following easy lemmas.

\begin{lemma}\label{lem:sums} For $N$, $i$,and $j$ we have
\begin{equation}\label{single_sum}
\sum_k{N-k\choose i}{k\choose j}={N+1\choose i+j+1}
\end{equation}
\begin{equation}\label{double_sum}
\sum_k{k\choose l}{N-k\choose i}{N-k-i\choose j} ={i+j\choose i}{N+1\choose i+j+l+1}.
\end{equation}
\end{lemma}
\begin{proof}
The first assertion is \cite[formula (5.26)]{GKP} with $q=0$. 
For the second formula note that
\[
{N-k\choose i} {N-k-i\choose j}
=\frac{(N-k)!}{i!(N-k-i)!}  \frac{(N-k-i)!}{j!(N-k-i-j)!}
={i+j\choose i} {N-k\choose i+j}.
\]
Hence, by \eqref{single_sum}, the left--hand side of \eqref{double_sum} is
\[{i+j\choose i}\sum_k{k\choose l}{N-k\choose i+j}={i+j\choose
  i}{N+1\choose i+j+l+1}. \qedhere
\]
\end{proof}

The next lemma follows from simple algebra and the proof is omitted.
\begin{lemma}\label{lem:fij}
For $
f_{xy}:=6{x+y-1\choose y}
+3{x+y\choose y+1}
+ {x+y+1\choose y+2}
$, we have 
\[
f_{ij}+f_{ji} = \frac{(i+j+4)!}{(i+2)!(j+2)!}. 
\] 
\end{lemma}

\subsection{Expected number of blocks of a given size}

It has been shown by
Velleman and Warrington that the expected value of  the number of
blocks of size $i$ is  asymptotic to $4n/(i+2)_3$ as the number of
pairs of cards $n$ goes to infinity. We will need a more
quantitative version of their result.

\begin{proposition}\label{lem:E[B_i]<>} 
As $n\to\infty$ and for $i=o(n^{1/2})$
\begin{equation}\label{E[B_i]>}
\mean{\bni}
=\mean{\xi_i}+\frac{4n}{(i+2)_3}
\left(1+O\left(\frac{i^2}n\right)\right),
\end{equation}
where $\xi_i$ is the indicator of the event that the first block has size $i$. 
(Note, in particular, that $\xi_1=0$.) The bound holds uniformly over $i^2/ n$.
Furthermore, for every $n\ge1$ and $1\le i\le n +1$, 
\begin{equation} \label{E[B_i]<}
\mean{\bni}
\le
\mean{\xi_i}+\frac{4e^{1/2}n}{(i+2)_3}.
\end{equation}
\end{proposition}

\proof
We use the set of permutations $\T_n$ throughout the proof. 
For $x\in [n]$, let $x_B$ and $x_R$ denote the blue and the red cards labeled with $x$, respectively. Recall that the block sizes are determined by the positions of the red cards.
For $a\not = b$, let $\iab$ denote the indicator random variable which
takes the value 1 if the following hold in the random permutation of
cards: (i) $ a_R$ comes before $b_R$, (ii) there is no red card
between $a_R$ and $b_R$, and (iii) there are exactly $i-1$ blue cards
between $a_R$ and $b_R$. If all three of these events hold, then the
contribution of the pair $(a_R,b_R)$ to $B_{n,i}$ is 1. 
Thus, we have
\begin{equation}\label{b_i}
\bni= \xi_i+\sum_{a\not =b}\iab.
\end{equation}
and
\begin{equation}\label{mean Bni=}
\mean{\bni}= \mean{\xi_i}+\sum_{a\not= b} \mean{\iab} = \mean{\xi_i}+ n(n-1)\mean{\ione},
\end{equation}
where the last identity follows from symmetry. 
For computing $\mean{\ione}$, we can first choose the positions of $1_R$ and $2_R$, then choose cards (all blue) to put between $1_R$ and $2_R$, and finally permute everything and divide by the total number of configurations. We write
\[
\mean{\ione}=E_1+E_2,
\]
where $E_1$ is the contribution of those configurations such that both $1_B$ and $2_B$ appear before $1_R$, and $E_2$ is the rest. Of course, $E_2$ is positive if and only if $2\le i\le n$.
 Now, as explained below,
\begin{align*}
E_1&=\sum_{k\ge 2} k(k-1) {n-2\choose i-1} (i-1)! {2n-k-i-1\choose i-1}(i-1)!\frac{(2n-2i-2)!/2^{n-i-1}}{(2n)!/2^n}.
\end{align*}
Explanation:  In the sum, $k+1$ represents the position of $1_R$. This implies that the position of $2_R$ is $k+i+1$. We choose the positions of $1_B$ and $2_B$ in $k(k-1)$ ways. We choose $i-1$ cards out of $n-2$ blue cards to put in between $1_R$ and $2_R$, and permute these in $(i-1)!$ ways. Next, we  choose $(i-1)$ positions from $\{k+i+2,\dots,2n\}$ for the red partners of the blue cards sandwiched between $1_R$ and $2_R$, and permute these red partners in $(i-1)!$ ways. The rest of the numerator is for permuting all the remaining cards.

After  simple cancellations, we get
\[
E_1= \frac{2^{i+2}(i-1)!(n-2)_{(i-1)}}{(2n)_{(2i+2)}}\sum_{k\ge 2}{k\choose 2}{2n-k-i-1\choose i-1}.
\]
Using~\eqref{single_sum} for the right--hand side above and simplifying the resulting expression, we get
\begin{equation}\label{E[B_i], E_1}
E_1=\frac{2^{i+2}(i-1)!(n-2)_{(i-1)}}{(2n)_{(2i+2)}} {2n-i\choose i+2}= \frac{2^{i+2}}{(i+2)_3}
 \cdot \frac{(n-2)_{(i-1)}}{(2n)_i}.
\end{equation}
Similarly, for $i\ge 2$ we have
\begin{align}\label{E[B_i], E_2}
E_2 &= \sum_{k\ge 1} k {n-2\choose i-2}(i-1)!{2n-k-i-1\choose i-2}(i-2)!\frac{(2n-2i)!/2^{n-i}}{(2n)!/2^n} \notag	\\
&=\frac{2^{i}(i-1)!(n-2)_{(i-2)}}{(2n)_{(2i)}}\sum_{k\ge 2} {k\choose 1} {2n-k-i-1\choose i-2}	\notag	\\
&=\frac{2^{i}(i-1)!(n-2)_{(i-2)}}{(2n)_{(2i)}} {2n-i \choose i}
=\frac{2^{i}}{i}\frac{(n-2)_{(i-2)}}{(2n)_i}.
\end{align}

Combining~\eqref{E[B_i], E_1} and~\eqref{E[B_i], E_2}, we get
\begin{equation}\label{E_1+E_2}
n(n-1)(E_1+E_2)= \frac{2^{i+2}\,
 (n)_{i+1}}{(i+2)_3\,(2n)_i}+\frac{2^i\, (n)_i}{i(2n)_i}
 =\frac{4n}{(i+2)_3}\frac{2^{i}\,
 (n)_{i}}{(2n)_i}
 \left(1+\frac{i^2-i+2}{4n}\right).
\end{equation}
It follows from Lemma~\ref{lem:fallfact} that for $i=o(\sqrt n)$
\[\frac{2^{i}\,
 (n)_{i}}{(2n)_i}=
 \left(1+O\left(\frac{i^2}{n}\right)\right)\]
 uniformly over $i^2/n$. This, combined with~\eqref{mean Bni=} and~\eqref{E_1+E_2},  implies \eqref{E[B_i]>}.

Furthermore, for $1\le i\le n+1$
\[\frac{2^i(n)_i}{(2n)_i}=\prod_{j=1}^{i-1}\frac{1-\frac jn}{1-\frac j{2n}}=\prod_{j=1}^{i-1}\left(1-\frac j{2n-j}\right)\le\prod_{j=1}^{i-1}\left(1-\frac j{2n}\right)\le\exp\left\{-\frac{i(i-1)}{4n}\right\}.
\]
Hence, applying $1+x\le e^x$ to the last factor in \eqref{E_1+E_2}
we obtain
\[
\frac{2^{i}\,
 (n)_{i}}{(2n)_i}
 \left(1+\frac{i^2-i+2}{4n}\right)\le\exp\left\{-\frac{i(i-1)}{4n}+\frac{i^2-i+2}{4n}\right\}\le e^{1/2},
\]
which proves \eqref{E[B_i]<}. 
\endproof

\subsection{Bound on the variance of the number of blocks of a given size}

In this section we estimate $\var(\bni)$. The following is sufficient
for our purposes.
\begin{lemma}\label{lem:varB_i<} There exists an absolute constant $M$ such that for
  $i=o(n^{1/3})$
\begin{equation}\label{varB_i<}
\var(B_{n,i}) \le M\,\frac{n}{i^3}.
\end{equation}
\end{lemma}
\proof By \eqref{b_i},
\begin{align}\label{def E[Bi2]}
\mean{\bni^2}&= \mean{\Big(\xi_i+ \sum_{(a,b)} \iab \Big)^2} \notag  \\
&= \mean{\xi_i}+\mean{\sum_{((a,b),(c,d))} \iab\cdot  \icd }+2\mean{\sum_{(a,b,c)} \iab\cdot  \ibc }+ 2\mean{ \xi_i \cdot \sum_{(a,b)} \iab}	\notag	\\
&=\mean{\xi_i}+ \Sigma_1+2\Sigma_2+2\Sigma_3, 
\end{align}
where different letters in the sums  denote distinct numbers from
$[n]$ and $\Sigma_k$ denotes the $k$-th sum before the last equality
sign. We will see that the main contribution to $\mean{B_{n,i}^2}$ comes from $\Sigma_1$.

By Proposition~\ref{lem:E[B_i]<>}, 		
 there are some absolute positive  constants $K_1$ and $K_2$ such that $\mean{\bni}$  lies in the interval $[K_1n/i^{3}, K_2n/i^3]$ for all $i=o(n^{1/3})$ as $n\to \infty$. Also, since there are only $n$ blocks, we have the trivial upper bound $\bni \le n$. Hence
\begin{equation}\label{Sigma3}
2\Sigma_3=2\mean{\xi_i(\bni-\xi_i)}\le 2n\mean{\xi_i}\le 2i,
\end{equation}
since, using~\eqref{probD_1=t}, we have
\begin{align*}
\mean{\xi_i}=\pr{D_{n,1}=i} = \frac{i-1}{2n-i+1}\cdot \frac{2^{i-1}(n)_{i-1}}{(2n)_{i-1}} \le \frac in
\end{align*}
for $i\le n+1$.
Next we estimate $\Sigma_1$. 

\medskip

\noindent \textbf{Computation of $\Sigma_1$.} Let $A= A(n,i)$ be the following event:
\vspace{-\topsep}
\begin{enumerate}\itemsep0pt \parskip0pt \parsep0pt \partopsep0pt \topsep0pt
\item[(a)]  $\{1_R,2_R,3_R,4_R\}$  appear in the order $1_R,2_R,3_R,4_R$ in the permutation,
\item[(b)] $1_R$ and $2_R$ are separated by $i-1$ blue cards,
\item[(c)] $3_R$ and $4_R$ are separated by $i-1$ blue cards.
\end{enumerate}
Let $p_i$ be the probability of $A$. Note that, for each permutation in $A$, the contribution of the 4-tuple $(1_R,2_R,3_R,4_R)$ to $\Sigma_1$  is 2, and we have
\[
\Sigma_1= 2{n \choose 4}4!p_i= 2(n)_4 \,p_i.
\]
We now estimate $p_i$ for $i=O(n^{1/3})$.
Before we start, we introduce some notation. The positions of the red cards $1_R$, $2_R$, $3_R$, and $4_R$ determine five disjoint intervals in a permutation. We denote the $j$-th interval by $I_j$ for $1\le j\le 5$. Depending on which interval(s) the blue cards $\{1_B,2_B,3_B,4_B\}$ lie, we have different cases. Let $t_u$ be the number of blue cards from $\{1_B,2_B,3_B,4_B\}$ in the interval $I_u$ for $1\le u \le 4$. Note that $I_5$ cannot contain any of these four blue cards. Since $x_B$ must appear before $x_R$ for any $x\in [n]$, we have the following constraints for the nonnegative integers $t_1$, $t_2$, $t_3$, and $t_4$.
\begin{equation}\label{t constraints}
\sum_{u=1}^{v}t_u \ge v, \ \text{for } v \in \{1,2,3\}; \quad t_1+t_2+t_3+t_4=4.
\end{equation}
Both intervals $I_2$ and $I_4$ have lengths $i-1$, that is, there are $i-1$ positions in each interval. We denote by $k$ and $\ell$ the lengths of $I_1$ and $I_3$, respectively. Thus, the length of $I_5$ is $(2n-4)-k-\ell-2(i-1)= (2n-2i-2-k-\ell)$. For a given $\bs t=(t_1,t_2,t_3,t_4)$ meeting the conditions in~\eqref{t constraints}, let $K(\bs t)$ be the number of allocations of  $\{1_B,\dots,4_B\}$ into intervals $I_1$, $I_2$, $I_3$,  and $I_4$, respecting $\bs t$. By considering the number of blue cards in intervals and the red partners of the blue cards in $I_2$ and $I_4$, we get
\begin{align*}
p_i&=\sum_{\bs t} K(\bs t) \sum_{k,\ell} {k \choose t_1}t_1!{i-1 \choose t_2}t_2!{\ell \choose t_3}t_3!{i-1 \choose t_4}t_4! {n-4 \choose 2i-2-t_2-t_4}(2i-2-t_2-t_4)!\\
&\qquad \times{2n-2i-k-\ell-2 \choose i-1-t_4}(i-1-t_4)!{2n-3i-k-1-t_3+t_4 \choose i-1-t_2}(i-1-t_2)!\\
&\qquad \times\frac{(2n-4i-4+2t_2+2t_4)!}{2^{n-2i-2+t_2+t_4}}\cdot \frac{2^n}{(2n)!}.
\end{align*}
After cancellations and taking the factors that do not depend on $k$ or $\ell$ out of the inner sum, we get
\begin{align*}
p_i&=\sum_{\bs t} K(\bs t) (i-1)!^2\,t_1!t_3! (n-4)_{2i-2-t_2-t_4}\cdot  \frac{2^{2i+2-t_2-t_4}}{(2n)_{4i+4-2t_2-2t_4}}\\
&\quad \times  \sum_{k} {k \choose t_1}{2n-3i-k-1-t_3+t_4 \choose i-1-t_2}\sum_{\ell}{\ell \choose t_3} {2n-2i-k-\ell-2 \choose i-1-t_4}.
\end{align*}
Using Lemma \ref{lem:sums},
\[
\sum_{\ell}{\ell \choose t_3} {2n-2i-k-\ell-2 \choose i-1-t_4}= {2n-2i-k-1 \choose i+t_3-t_4},
\]
so we need the sum
\begin{equation}\label{sum over k}
 \sum_{k} {k \choose t_1}{2n-3i-k-1-t_3+t_4 \choose i-1-t_2}{2n-2i-k-1 \choose i+t_3-t_4}.
\end{equation}
For the product of the second and the third terms in the sum, we write
\[
{2n-3i-k-1-t_3+t_4 \choose i-1-t_2}{2n-2i-k-1 \choose i+t_3-t_4}
= {2i-1-t_2+t_3-t_4 \choose i-1-t_2} {2n-2i-k-1 \choose 2i-1-t_2+t_3-t_4}.
\]
Using this identity, the sum in \eqref{sum over k} can be written as
\[
{2i-1-t_2+t_3-t_4 \choose i-1-t_2}\sum_{k} {k \choose t_1}{2n-2i-k-1 \choose 2i-1-t_2+t_3-t_4}.
\]
Using Lemma \ref{lem:sums} once again, this becomes
\[
{2i-1-t_2+t_3-t_4 \choose i-1-t_2}{2n-2i \choose 2i+t_1-t_2+t_3-t_4}.
\]
Combining the results in the previous lines and using the identity $4-2t_2-2t_4= t_1-t_2+t_3-t_4$,
\begin{align}\label{p_ij times (n)_4}
(n)_4\cdot p_i&=\sum_{\bs t} K(\bs t) (i-1)!^2\,t_1!t_3!\, (n)_{2i+2-t_2-t_4}	  \notag \\
&\qquad \times  \frac{2^{2i+2-t_2-t_4}} {(2n)_{4i+4-2t_2-2t_4}} {2i-1-t_2+t_3-t_4 \choose i-1-t_2}{2n-2i \choose 2i+t_1-t_2+t_3-t_4} \notag \\
&= \sum_{\bs t} K(\bs t)t_1!t_3! \, \frac{(i-1)!^2}{(i-1-t_2)!(i+t_3-t_4)!}
\, \frac{(2i-1-t_2+t_3-t_4)!}{(2i+t_1-t_2+t_3-t_4)!}
\, \frac{2^{2i+2-t_2-t_4}\cdot (n)_{2i+2-t_2-t_4}}{(2n)_{2i}}.
\end{align}
Defining 
$a:=2i+2-t_2-t_4 $ 
and using Lemma~\ref{lem:fallfact} for the falling factorials in \eqref{p_ij times (n)_4}, we obtain
\begin{align*}
(n)_4\, p_i&=\sum_{\bs t}K(\bs t)t_1!t_3! \ \frac{(i-1)!^2}{(i-1-t_2)!(i+t_3-t_4)!} 
\ \frac{(2i-1-t_2+t_3-t_4)!}{(2i+t_1-t_2+t_3-t_4)!}		\notag	\\
&\quad \times (2n)^{t_1+t_3-2}\ \e{\frac{-2{a \choose 2}+{2i\choose 2}}{2n}} \left[1+ O\left({\frac{1}{n}}\right) \right],
\end{align*}
where
\[
\frac{-2{a \choose 2}+{2i\choose 2}}{2n} = -\frac{i^2}{n}+O\left( \frac{i}{n}\right).
\]
Hence, for $i=o(n^{1/3})$, we have
\begin{equation} \label{n4pi}
(n)_4\, p_i = \left( 1+O(i^2/n)\right)\sum_{\bs t} K(\bs t)t_1!t_3!\,  (2n)^{t_1+t_3-2}\,F(\bs t),
\end{equation}
where
\begin{equation}\label{def,F}
 F(\bs t) = \frac{(i-1)!^2}{(i-1-t_2)!(i+t_3-t_4)!} 
\ \frac{(2i-1-t_2+t_3-t_4)!}{(2i+t_1-t_2+t_3-t_4)!}	.
\end{equation}
Note that there exists an absolute constant $\gamma$ such that for any $i\ge 1$ and $\bs t$, 
\[
F(\bs t) \le \gamma \, i^{-1-t_3+t_2+t_4}\, i^{-1-t_1}=\gamma\, i^{2-2t_1-2t_3}.
\]
Consequently, the sum in~\eqref{n4pi} over the vectors $\bs t=(t_1,t_2,t_3,t_4)$ with $t_1+t_3\le 3$ is only $O(n/i^4)$. 
Also, there are three vectors $\bs t$ with the sum of the first and the third components being 4:
$\bs t_1=(4,0,0,0)$, $\bs t_2=(3,0,1,0)$, and $\bs t_3=(2,0,2,0)$. For these $\bs t$ we have $K(\bs t_1)=1$, $K(\bs t_2)=2$, and $K(\bs t_3)=1$. 
Then,
\begin{equation}\label{n4pi-2}
(n)_4\,p_i = O\left(n/i^4\right)+  \left( 1+O\left(i^2/n\right)\right)(2n)^2 \Big( 24 F(\bs t_1)+12F(\bs t_2)+4F(\bs t_3)\Big).
\end{equation}
Using~\eqref{def,F}, we have
\[
24F(\bs t_1)+12F(\bs t_2)+4F(\bs t_3)	= 4\,\frac{(i-1)!^2}{(2i+4)!}\left[6{2i-1 \choose i}+3{2i\choose i+1}+{2i+1 \choose i+2} \right].			
\]
Using Lemma~\ref{lem:fij} with $i=j$, we get
\[
\left[6{2i-1 \choose i}+3{2i\choose i+1}+{2i+1 \choose i+2} \right] =\frac{(2i+4)!}{2(i+2)!^2},
\]
and consequently,
\[
24F(\bs t_1)+12F(\bs t_2)+4F(\bs t_3)	= \frac{2}{\big((i+2)_3\big)^2}.
\]
Finally, using this equation and~\eqref{n4pi-2}, we get
\begin{equation}\label{Sigma1}
\Sigma_1 =2(n)_4\,p_i= O\left(\frac{n}{i^4}\right)+ \left(1+O\left(\frac{i^2}{n}\right)\right)
\frac{16n^2}{\big((i+2)_3\big)^2}
= \frac{16n^2}{\big((i+2)_3\big)^2}  +  O\left(\frac{n}{i^4}\right).
\end{equation}

\smallskip

\noindent {\bf Computation of $\Sigma_2$.} Let $A'$ denote the event $\{I_{1,2}^{i}=1=I_{2,3}^{i}\}$ in a game with $n$ pairs of cards and let $q=\pr{A'}$.
Thus we have $\Sigma_2= (n)_3\cdot q$ and we need to estimate $q$. On the event $A'$, $1_R$ appears before $2_R$, which appears before $3_R$. These three red cards determine four intervals, $I_1$ through $I_4$, numbered in the order they appear. Both $I_2$ and $I_3$ contain $i-1$ blue cards and no red cards.  For $1\le j\le 3$, let $t_j$ denote the number of elements from $\{1_B,2_B,3_B\}$ contained in $I_j$.
We have the following possibilities for $\bs t=(t_1,t_2,t_3)$:
\[
\bs t_1=(3,0,0); \quad \bs t_2=(2,1,0); \quad \bs t_3=(2,0,1); \quad \bs t_4=(1,2,0); \quad \bs t_5=(1,1,1).
\]
Let $K_u$ be the number of ways to place $\{1_B,2_B,3_B\}$ in $I_1$, $I_2$, and $I_3$ respecting $\bs t_u$. We have
$K_1=1$, $K_2=2$, $K_3=1$, $K_4=1$, and $K_5=1$. Let $k$ denote the length of $I_1$, so the length of $I_4$ is $2n-2i-k-1$. 
We need to put $2i-2-t_2-t_3$ blue cards other than $\{1_B,2_B,3_B\}$ in $I_2\cup I_3$ and we need to put the red counterparts of these blue cards in $I_4$. Hence, 
using Lemma \ref{lem:sums} in the last step below, we have
\begin{align*}
q&=\sum_{u=1}^5\sum_{k} K_u {k \choose t_1}{i-1 \choose t_2}{i-1 \choose t_3}t_1!t_2!t_3! {n-3 \choose 2i-2-t_2-t_3}(2i-2-t_2-t_3)!\\
&\qquad \qquad \times \ {2n-k-2i-1 \choose 2i-2-t_2-t_3} (2i-2-t_2-t_3)!
\ \frac{(2n-4i-2+2t_2+2t_3)!}{2^{n-2i-1+t_2+t_3}} \cdot \frac{2^n}{(2n)!}\\
&=(i-1)!^2\ \sum_{u=1}^{5}K_u t_1! {2i-2-t_2-t_3 \choose i-1-t_2}  (n-3)_{2i-2-t_2-t_3}\\
&\qquad \qquad \qquad \times \frac{2^{2i+1-t_2-t_3}}{(2n)_{4i+2-2t_2-2t_3}}\cdot \sum_{k} {k \choose t_1} {2n-k-2i-1 \choose 2i-2-t_2-t_3}\\
&=(i-1)!^2\ \sum_{u=1}^{5}K_u t_1! {2i-2-t_2-t_3 \choose i-1-t_2}  (n-3)_{2i-2-t_2-t_3} \\
& \qquad \qquad \qquad \times \frac{2^{2i+1-t_2-t_3}}{(2n)_{4i+2-2t_2-2t_3}}\cdot  {2n-2i \choose 2i-1+t_1-t_2-t_3}.
\end{align*}
 Now, using 
\[ 	 2i-1+t_1-t_2-t_3= 2i+2-2t_2-2t_3,		\]
we simplify the last expression and write
\begin{align*}
q&= \sum_{u=1}^{5}K_u t_1! (i-1)!^2\ {2i-2-t_2-t_3 \choose i-1-t_2}  (n-3)_{2i-2-t_2-t_3}\\
 &\qquad \times \frac{2^{2i+1-t_2-t_3}}{(2n)_{2i}}\cdot \frac{1}{(2i+2-2t_2-2t_3)!}\\
&=\sum_{u=1}^{5}O\left( \frac{1}{i^{4-2t_2-2t_3}}\cdot \frac{2^{2i}\, (n)_{2i}}{(2n)_{2i}} \cdot \frac{1}{n^{2+t_2+t_3}}\right) =\sum_{u=1}^{5}O\left( \frac{1}{i^{4-2t_2-2t_3}}\cdot\frac{1}{n^{2+t_2+t_3}} \right){.}
\end{align*}
In the sum above, the main contribution comes from $u=1$, i.e.\ when $(t_1,t_2,t_3)=(3,0,0)$, in which case we have $O(n^{-2}i^{-4})$. In other words, we have
$q= O(n^{-2}i^{-4})$. Then, 
\begin{equation}\label{Sigma2}
\Sigma_2= (n)_3\cdot q= O\left(\frac{n}{i^4}\right).
\end{equation}
Combining 
\eqref{def E[Bi2]}, \eqref{Sigma3}, \eqref{Sigma1}, and
\eqref{Sigma2}, we get \eqref{varB_i<} for some absolute constant $M$. \endproof

As a corollary of the previous two lemmas, we get the following concentration result for
 $B_{n,i}$'s. 
Recall that $\ov{B}_{n,i}=B_{n,i}-\mean{B_{n,i}}$ for $i\ge 1$.

\begin{corollary}\label{cor:concentration B_i}
Let $l=o(n^{1/3})$ and let $\omega\to \infty$. 
Then
\[
\pr{|\ov{B}_{n,1}| \le \sqrt{n\omega}\,,\dots\,,|\ov{B}_{n,\ell}| 
\le \sqrt{n\omega} } =1-O(1/\omega).
\]
\end{corollary}

\begin{proof}
Using Lemma~\ref{lem:varB_i<}, and Chebyshev's inequality, for any 
$i \le l$, we get
\[
\pr{\,\big|\ov{B}_{n,i}\big| \ge \sqrt{n\omega}\,} \le {M}/(i^3\omega),
\]
where $M$ is a  constant from Lemma~\ref{lem:varB_i<}.
Hence
\[
\sum_{i=1}^\ell  \pr{\,\big|\ov{B}_{n,i}
\big| \ge \sqrt{n\omega}\,} = O(1/\omega),
\]
from which the corollary follows.
\end{proof}

\section{Lucky moves}

The asymptotics of the expected value of lucky moves in a random game of size $n$ was derived in~\cite[Corollary~18]{VW}. Here we prove Theorem~\ref{thm:lucky}, which gives the limiting distribution of the number of lucky moves as $n\to \infty$. 

\begin{proof}[Proof of Theorem~\ref{thm:lucky}]
Let us write $L$ instead of $L_n$ for the simplicity of notation. We will compute, asymptotically,  the factorial moments $\mean{(L)_k}$ for $k\ge 1$. Fix $k$ and let $\mathcal E$ be the event that the cards numbered $n-k+1,\dots,n$ constitute lucky moves in a random deal in $\T_n$.  Clearly,
\[
\mean{(L)_k} = (n)_k\cdot \pr{\mathcal E} \sim n^k\cdot \pr{\mathcal E}.
\]
Thus, we need to find $\pr{\mathcal E}$ asymptotically. To this end, we use the following algorithm to generate a random deal with $n$ pairs of cards. 

Start with an uncolored deck of cards and put the cards in a row one by one in the following way. At the first round put one of the two cards labeled $1$ in a row. For $k\ge 1$, at round $k+1$, insert a card with label $\lceil (k+1)/2\rceil$ into an interval chosen uniformly at random from the $(k+1)$ intervals determined by the first $k$ cards already in the row, where an interval refers to a space either between two cards, or before the first card, or after the last card. After a pair of identical cards are inserted, which happens after each even-numbered step, color the one on the left with blue and the one on the right with red. We denote by $(a_1,\dots,a_k)$ the permutation after  the $k$-th step of this algorithm, so $(a_1,\dots,a_{2n})$ denotes a random game from $\T_n$.

After $m$ pairs of cards are put down, there are $2m+1$ intervals. Let them be $I_1,\dots,I_{2m+1}$ from left to right. For $m\ge 1$, let $S_0^{(m)}$ denote the set of intervals following red cards in $(a_1,\dots,a_{2m})$ together with the first interval, that is,
\[
S_0^{(m)}:= \{I_{i}: a_{i-1} \text{ is red} \} \cup \{I_1\}.
\]
Recursively, for $t\ge 1$, define
\[
S_t^{(m)}:= \{I_i: a_{i-1} \text{ is blue and } i-1\in S_{t-1}^{(m)} \}.
\]
Let $S^{(m)}:= \cup_{i\ge 0} S_{2i}^{(m)}$ be the set of `good' intervals. A good interval is an interval such that if two identical cards are inserted into this interval, then those two cards constitute a lucky move. Note that the intervals are alternately good and bad as long as a red card does not interfere. The sets $S_i^{(m)}$ partition the set $[2m+1]$ and 
\[
|S^{(m)}| \ge |S_0^{(m)}| =m+1.
\]
Event $\mathcal E$ is equivalent to the following event. Recursively, for $1\le j\le k$, the two $(n-k+j)$'s are inserted next to each other in the algorithm and the first one is inserted into a good interval, i.e.\  into an interval in $S^{(n-k+j-1)}$. Note that, after such an insertion, both the number of good intervals and the number of bad intervals increase by one, i.e.\ $|S^{(n-k+j)}|=|S^{(n-k+j-1)}|+1$. Let $s$ denote the number of good intervals after $n-k$ pairs of cards are placed, that is, $s=|S^{(n-k)}|$. Since $s\ge n-k+1$, conditionally on $s$, the probability of $\mathcal E$ is
\[
\prod_{j=0}^{k-1} \frac{s+j}{{2n-2k+2j+1 \choose 2}+{2n-2k+2j+1}} \sim \left(\frac{s}{2n^2}\right)^k,
\]
where the numerator on the left--hand side is the number of good intervals and the denominator is the number of ways to insert the two cards with labels $n-k+j+1$. 
Since $|S_0^{(n-k)}|= n-k+1$ and $|S_{2i}^{(n-k)}| =\sum_{j\ge 2i+1} B_{n-k,j}\,$ for $i\ge 1$, we have
\begin{equation}\label{eq:s=}
s=\sum_{i\ge 0} |S_{2i}^{(n-k)}| = 1+(n-k) + \sum_{i\ge 1}\sum_{j\ge 2i+1}  B_{n-k,j}.
\end{equation}
Hence, by Proposition~\ref{lem:E[B_i]<>}
  and the identity
\begin{equation}\label{par_frac}
\sum_{j\ge 2i+1}\frac{4}{(j+2)_3}  
=   \sum_{j\ge 2i+1}\lp \frac2j-\frac4{j+1}+\frac2{j+2} \rp =\frac{2}{2i+1}-\frac{2}{2i+2},
\end{equation}
we get 
\begin{align*}
\mean{s} &= n-k+1+ \sum_{i\ge 1}\sum_{j\ge 2i+1} \mean{ B_{n-k,j}}\  \sim \  n+ \sum_{i \ge 1}\sum_{j\ge 2i+1} \frac{4n}{(j+2)_3} \\
&=n+ n\sum_{i\ge 1} \left(\frac{2}{2i+1}-\frac{2}{2i+2}\right)\\
&= 2n\sum_{i\ge 0} \left[\frac{1}{2i+1}-\frac{1}{2i+2}\right]= (2\ln2)n.
\end{align*}
Next we show that $s/\mean{S}\to 1$ in probability. For this, it is enough to show that for every $\eps>0$
\[
\bbP{|s-\mean{s}|>\eps n}\to0,\quad\mbox{as\ }n\to\infty.
\]
This probability is bounded by the sum
\[
\bbP{\Big|\sum_{i=1}^l\sum_{j\ge 2i+1}\ov{B}_{n-k,j}\Big|>\frac{\eps n}2}+ 
\bbP{\Big|\sum_{i>l}\sum_{j\ge 2i+1}\ov{B}_{n-k,j}\Big|>\frac{\eps n}2},
\]
where, for a random variable $X$, we put $\ov{X}=X-\mean{X}$.
Since $\sum_jB_{m,j}=m$ for every $m\ge1$,  we have
\begin{align*}
\mathbb{P}\Big(\, \Big|\sum_{i=1}^l\sum_{j\ge 2i+1}\ov{B}_{n-k,j}\Big|>\frac{\eps n}2\, \Big)
&=\mathbb{P}\Big(\, \Big|\sum_{i=1}^l\sum_{j=1}^{2i}\ov{B}_{n-k,j}\Big|>\frac{\eps n}2\, \Big) 	
&\le\bbP{\exists\ 1\le j\le2l:\ |\ov{B}_{n-k,j}|>\frac{\eps n}{2l(l+1)}}.
\end{align*}
By Corollary~\ref{cor:concentration B_i} applied with $\omega=\eps^2n/(2l(l+1))^2$  the right--hand side above goes to zero  for $l=o(n^{1/4})$. 

Now consider the second probability. Since 
$\bbP{|\ov{X}|>t}\le2\mean{X}/t$ 
for a non--negative random variable $X$ and any $t>0$, using Proposition~\ref{lem:E[B_i]<>} we get
\begin{align*}
\bbP{\Big|\sum_{i>l}\sum_{j\ge 2i+1}\ov{B}_{n-k,j}\Big|>\frac{\eps n}2}
&\le  \frac4{\eps n}\sum_{i>l}\sum_{j\ge2i+1}\mean{B_{n-k,j}}
\le \frac4{\eps n}\sum_{i>l}\sum_{j\ge2i+1}\left(\mean{\xi_j}+\frac {4e^{1/2}n}{(j+2)_3}\right) \\
&\le\frac4{\eps n}\sum_{i>l}\left(\bbP{D_{n-k,1}\ge 2i+1}+\frac {cn}{i^2}\right)
\le \frac4{\eps n}\left(\mean{D_{n-k,1}}+\frac{cn}l\right).
\end{align*}
By Proposition~\ref{prop:A_recur}, $\mean{D_{n-k,1}}=O(\sqrt n)$ so that the right--hand side goes to $0$ as long as $l\to\infty$ with $n\to\infty$.
Thus
\[
\mean{(L)_k} \sim n^k \cdot \left(\frac{(2\ln2) n}{2n^2}\right)^k = (\ln2)^k,
\]
and $L$ converges in distribution to $Pois(\ln 2)$.

\end{proof}

\section{Length of the game}
In this section we prove
Theorem~\ref{thm:main} and thus also the asymptotic normality of the length of the
game (see Corollary~\ref{cor:G_n} in Section~\ref{sec:main}). 
The proof is an application of Theorem~\ref{thm:joint} which will be
discussed in the next section.

We first note that by  using~\eqref{mean Bni=}, \eqref{E_1+E_2}, and the fact that $\sum_{i\ge 1}\mean{\xi_i}=1$, 
and by letting $i_n$ be an integer of order $\sqrt{n/\log n}$ we have
 \begin{align*}\sum_{i\ge1}\mean{B_{n,2i}}&=
 \sum_{i=1}^{i_n}\mean{B_{n,2i}}+\sum_{i> i_n}\mean{B_{n,2i}}=\sum_{i=1}^{i_n}\left(\mean{\xi_{2i}}+
 \frac{4n}{(2i+2)_3}+O\left(\frac1i\right)\right)+O\left(\sum_{i>i_n}\frac n{i^3}\right)\\
 &=\sum_{i=1}^\infty\frac{4n}{(2i+2)_3}+O(\log n)+O\left(\frac n{i_n^2}\right)=
 ( 3-4\ln 2)n+O\left(\log n\right),
\end{align*}
where the last equality follows by the partial fraction decomposition used in~\eqref{par_frac}. Hence,
\[
 \frac1{\sqrt n}\left( Y_n- (3-4\ln 2)n\right)  =  \frac1{\sqrt n}\sum_{i\ge 1}\ov B_{n,2i} +O\left(\frac{\log n}{\sqrt n}\right)
\]
and it suffices to show that the first term on the right--hand side is asymptotically normal.

For $j\ge 1$ let $V_j=\sum_{i=1}^j W_{2i}$ and $V=\sum_{i=1}^\infty W_{2i}$ where $(W_i)$ are defined  in Theorem~\ref{thm:joint}.
Pick any $\eps>0$. Pick $N_0$ such that  for every $N\ge N_0$ we have $\bbP{|V-V_N|>\eps/2}<\eps/3$. 
This is possible since the variance of $V-V_N$ goes to zero as $N\to\infty$ and hence $V-V_N$ goes to zero in probability. 
Take any $x\in \Bbb R$ and any $N\ge N_0$ and pick $n_0$ such that for $n\ge n_0$
\[
\left|\bbP{\frac1{\sqrt n}\sum_{i\le N}\ov B_{n,2i}\le
x+\eps/2}-\bbP{V_{N}\le x+\eps/2}\right|<\eps/3
\]
and
\[\left|\bbP{\frac1{\sqrt n}\sum_{i\le N}\ov B_{n,2i}\le
x-\eps/2}-\bbP{V_{N}\le x-\eps/2}\right|<\eps/3.
\]
This is possible since by Theorem~\ref{thm:joint} there is a joint convergence, and thus also
convergence of any finite sums.

Now, for $n\ge n_0$  we have
\begin{align*}
&\bbP{\frac1{\sqrt n}\sum_i\ov B_{n,2i}\le x}=\bbP{\frac1{\sqrt n}\sum_{i\le
N}\ov B_{n,2i}+\frac1{\sqrt n}\sum_{i>N}\ov B_{n,2i}\le x}\\&\le
\bbP{\frac1{\sqrt n}\sum_{i\le N}\ov B_{n,2i}\le
x+\eps/2}+\bbP{\frac1{\sqrt n}\Big|\sum_{i>N}\ov B_{n,2i}\Big|>\eps/2} \\&\le
\bbP{V_{N}\le x+\eps/2}+\eps/3+\bbP{\frac1{\sqrt n}\Big|\sum_{i>N}\ov
B_{n,2i}\Big|>\eps/2}
\\&\le
\bbP{V\le x+\eps}+2\eps/3+\bbP{\frac1{\sqrt n}\Big|\sum_{i>N}\ov B_{n,2i}\Big|>\eps/2}.
\end{align*}
Similarly, 
\begin{align*}
&\bbP{\frac1{\sqrt n}\sum_i\ov B_{n,2i}\le x}\ge\bbP{\frac1{\sqrt
n}\sum_i\ov B_{n,2i}\le x,\frac1{\sqrt n}\Big|\sum_{i>N}\ov B_{n,2i}\Big|\le \eps/2}\\&\ge
\bbP{\frac1{\sqrt n}\sum_{i\le N}\ov B_{n,2i}\le
x-\eps/2}-\bbP{\frac1{\sqrt n}\Big|\sum_{i>N}\ov B_{n,2i}\Big|>\eps/2} \\&\ge
\bbP{V_{N}\le x-\eps/2}-\eps/3-\bbP{\frac1{\sqrt n}\Big|\sum_{i>N}\ov
B_{n,2i}\Big|>\eps/2}
\\&\ge
\bbP{V\le x-\eps}-2\eps/3-\bbP{\frac1{\sqrt n}\Big|\sum_{i>N}\ov B_{n,2i}\Big|>\eps/2}.
\end{align*}
As we will show below, $n_0$ and $N_0$ can be chosen so that
\begin{equation}\label{B_tail}
\bbP{\frac1{\sqrt n}\Big|\sum_{i>N}\ov
B_{n,2i}\Big|>\eps/2}\le\eps/3\end{equation}
for  $n\ge n_0$ and $N\ge N_0$. 
We then  have
\[\bbP{V\le x-\eps}-\eps\le\bbP{\frac1{\sqrt n}\sum_i\ov B_{n,2i}\le
x}\le\bbP{V\le x+\eps}+\eps.
\]
Letting $\eps\to0$ and using the fact that the distribution of $V$ is
continuous we conclude that 
\[
\bbP{\frac1{\sqrt n}\sum_i\ov B_{n,2i}\le x}
\to\bbP{V\le x}.
\]
Finally, note that $V\stackrel d=N(0,\sigma^2)$ where
\begin{align*}\sigma^2&={\rm
var}\left(\sum_{i\ge1}W_{2i}\right)=\sum_{i\ge1}\sigma_{2i,2i}+2\sum_{1\le
i<j}\sigma_{2i,2j}\\&=4\sum_{i=1}^\infty\frac1{(2i+2)_3}+16\sum_{i,j=1}^\infty\frac1{(2i+2)_3(2j+2)_3}-24\sum_{i,j=1}^\infty\frac1{(2i+2j+4)_4}\\&=
(3-4\ln2)(4-4\ln2)-24\sum_{j=1}^\infty\frac{1}{(4j+4)_4}-48\sum_{j=2}^\infty\sum_{i=1}^{j-1}\frac1{(2i+2j+4)_4}\\&=
(3-4\ln2)(4-4\ln2)-24\left(-\frac1{24}+\frac{\ln2}4-\frac\pi{24}\right)-48\left(\frac{13}{24}+\frac\pi{48}-\frac{7\ln2}8\right)\\&=(4\ln2)^2+8\ln2-13.
\end{align*}
It remains to prove \eqref{B_tail}. Pick $i_n>N$ of order
$n^{(1/3)-\eta}$ where $0<\eta<1/12$.
 The 
left--hand side of \eqref{B_tail}  is bounded by
\begin{equation}\label{split}\bbP{\frac1{\sqrt n}\Big|\sum_{N<i\le i_n}\ov B_{n,2i}\Big|>\frac\eps4}+
\bbP{\frac1{\sqrt n}\Big|\sum_{i> i_n}\ov B_{n,2i}\Big|>\frac\eps4}.
\end{equation}
The first probability by Chebyshev's inequality 
is at most 
\[\frac{16}{\eps^2n}\var\left(\sum_{N<i\le i_n}B_{n,2i}\right)=
\frac{16}{\eps^2n}\left(\sum_{N<i\le i_n}\var(B_{n,2i})+
2\sum_{N<i<j\le i_n}\cov(B_{n,2i},B_{n,2j})\right).
\]
It follows from Lemma~\ref{lem:varB_i<} and the choice of $i_n$ that
 the first sum is  $O(n/N^2)$. By Cauchy--Schwartz
\[
|\cov(B_{n,2i},B_{n,2j})|\le \big(\var(B_{n,2i})\var(B_{n,2j})\big)^{1/2}=O\left(\frac
  n{i^{3/2}j^{3/2}}\right).
\]
Thus, the second sum is of order at most 
\[n\sum_{i>N}\frac1{i^{  3/2}}
\sum_{j>i}\frac1{j^{3/2}}\le
cn\sum_{i>N}\frac1{i^{3/2}}\frac1{i^{1/2}}=O\left(\frac nN\right)
\]
so that the first probability in~\eqref{split} is $O(1/N)$. 
The second probability, by Markov and \eqref{E[B_i]<} is bounded by 
\begin{align*}
&\frac4{\eps\sqrt n}\bbE{|\sum_{i>i_n}\ov B_{n,2i}|}\le
 \frac8{\eps\sqrt n}\bbE{\sum_{i>i_n}B_{n,2i}}
 \le\frac8{\eps\sqrt   n}
\sum_{i>i_n}\left(\bbE{\xi_i}+\frac{4e^{1/2}n}{i(i+1)(i+2)}\right)\\&\quad
=\frac8{\eps\sqrt n}\bbP{D_{n,1}>i_n}+O\left(\frac{\sqrt
    n}{i_n^2}\right)=O(1/\sqrt n)+O\left(n^{2\eta-1/6}\right)=o(1),
\end{align*}
by our choice of $i_n$. Thus, for a given $\eps>0$ each of the
probabilities in \eqref{split} can be made smaller than $\eps/6$ by
choosing $N_0$ and $n_0$ sufficiently large. This implies \eqref{B_tail}.

\section{Proof of Theorem~\ref{thm:joint}}
We consider the following dynamic version of the game. For $k\ge 1$, at step $k$, we insert the pair of cards labeled $k$ into the game so that the position of the blue card is chosen uniformly at random from $2k-1$ positions first and then the red card goes to the end of the row.
We can represent the evolution of blocks 
as a generalized P\'olya urn
model with infinitely many types of balls: $0,1,2,\dots$.  Drawing a ball of type $i\ge1$
corresponds to inserting the blue card 
of the new pair in an existing block of length $i$. 
Drawing a type zero ball represents a
situation when the blue and red cards 
 of  the arriving pair are both put
at the end thus creating a block of size 2. 

The evolution of the urn is as follows: we start with   one ball of type 0. 
If, at any time,  a ball of type $i$, $i\ge1$ is drawn it is replaced by one ball of type 1  and one  ball of type $i+1$. 
If a ball of type $0$ is drawn it is returned to the urn along with  one ball of type 2 and the process is continued in the same manner. 
After $n$ draws, there will be one ball of type 0 and $n$ balls of other types.
With this interpretation, the
number of blocks of length $i$ when $2n$ cards have been played,
$B_{n,i}$, is  the number of balls of type $i$ in the urn after the
$n$th draw and their asymptotic distributions have been studied widely in the literature.

In fact, if not for the ball of type zero it would be {\sl exactly} the
situation of outdegrees in a random plane recursive tree studied by
Janson in \cite{SJ_s} (see Theorem~1.3 in particular) and in a much more general
setting in~\cite {SJ_l} except that his $i+1$ (and $i+2$, $i+3$) are
our $i$, $i+1$, and $i+2$, respectively. 

To deal with this type 0 ball, we wish to  apply  Theorem~2.1 in  \cite{hu+}. 
In that work, a type 0 ball is referred to as an {\em immigration ball}. Let us call the vector $(a_{m,k})_{k=0}^K$ representing the number of balls added to the urn after an immigration ball is drawn the {\em immigration vector}. 
 Our immigration vector $(a_{m,k})_{k=0}^K$ (see item (a) on p.~646 in \cite{hu+}) is time
independent and is given by 
\[
[a_{m,0},\dots,a_{m,k}]=[0,0,1,0,\dots,0]. 
\]
Similarly, the addition (or replacement) matrices 
are time
independent,  and non--random, and the 
number of balls in the urn increases when a non-immigration ball is drawn. 
Thus, our situation falls into the second of the three cases described in \cite[Section~2.3]{hu+}  and in this case, Theorem~2.1 in~\cite{hu+} states that the immigration is asymptotically negligible and the urn has the same asymptotics as  without immigration.

The urn without immigration was studied by Janson~\cite{SJ_l,SJ_s} 
 who reduced infinitely many
types to finitely many and used a `superball trick' (see \cite[Remark~4.2]{SJ_l}).

 After these reductions, in Janson's 
 notation (and including the
immigration balls), if $\xi_i=(\xi_{0,i},\xi_{1,i},\dots,\xi_{M,i})$  is the replacement (column) vector when
a ball of type  $i$ is drawn, then with $\delta_{i,j}$ denoting the Kronecker delta,  
\[\xi_{0,j}=2\delta_{2,j}\quad\mbox{i.e.\ }
\xi_0=(0,0,2,0\dots,0).\]
For $1\le i<M$
\[\xi_{i,j}=\delta_{1,j}-i\delta_{i,j}+(i+1)\delta_{i+1,j}=(0,1,0\dots,0,-i,i+1,0,\dots,0),\]
where $-i$ is on the $(i+1)$st position (we start enumeration at
0). Finally, \[\xi_{M,j}=\delta_{1,j}+\delta_{M,j}=(0,1,0,\dots,0,1).\]
 The replacement matrix $A_0$
(which has vectors $\xi_j$ as its columns) is 
\[
A_0={\bf H'}=\left[\begin{array}{ccccccc}
0&0&0&0&\dots&0&0\\
0&0&1&1&\dots& 1&1\\
2&2&-2&0&\dots&0&0\\
&&\dots&&\dots&&\\
0&0&\dots&0&M-1&-M+1&0\\
0&0&\dots&0&0&M&1
\end{array}\right],
\]
where ${\bf H}$ was the notation used in \cite{hu+}. Notice that if $A$ is a
matrix obtained from $A_0$ by removing its first column and its first
row, then it is exactly the matrix considered by Janson \cite{SJ_s}
for the random plane recursive tree (with $i$ shifted by 1 since he
starts enumeration at $i=0$ and ours, after discarding balls of type
zero, starts at $i=1$). It follows from his argument  (see \cite[proof
of Theorem~1.3]{SJ_s} and Remark~(b) below) that the
eigenvalues of  $A$ are: $2,-1,\dots,-M$  and, consequently, 
 that \eqref{eqn:joint} holds with covariances of $(W_i)$ given  by
\[\sigma_{i,j}=2\sum_{k=0}^{i-1}\sum_{l=0}^{j-1}\frac{(-1)^{k+l}}{k+l+4}{i-1\choose
  k}{j-1\choose l}\left(\frac{2(k+l+4)!}{(k+3)!(l+3)!}-1-\frac{(k+1)(l+1)}{(k+3)(l+3)}\right).\]
 As was shown in \cite[Proposition~1]{AH} this expression simplifies
 to \eqref{eqn:cov}. This proves Theorem~\ref{thm:joint}.

\medskip
\noindent{\bf Remarks:}
\begin{itemize}\item[(a)] 
It is tempting to incorporate the ball of type 0 in the general framework of Janson, i.e.\ to consider
  $A_0$ as the replacement matrix. An additional eigenvalue is $0$ so we are still in the regime ${\rm
Re}(\lambda_2)<\frac12\lambda_1$ so that Janson's general theory \cite{SJ_l} would
apply. Accordingly,  
for the joint  distribution of $(B_{n,i})_{i\ge0}$ we  would have 
\[
\frac1n\left(B_{n,i}\right)_{i=0}^\infty\to\left(0,\left(\frac4{(i+2)_3} \right)_{i=1}^\infty\right),
\ a.s. \quad{\rm as \ }n\to\infty
\]
and  
\[
\frac1{\sqrt{n}}\left(B_{n,0}-1,\Big(B_{n,i}-\frac{4n}{(i+2)_3}\Big)_{i=1}^\infty\right)
\stackrel d\to( W_i)_{i=0}^\infty,\quad{\rm as \ } n\to\infty,
\]
where $(W_i)$ are jointly Gaussian with covariances $\sigma_{i,j}$ 
given by \eqref{eqn:cov} if  $i,j\ge1$ and   $\sigma_{ij}=0$ if 
$ i=0$ or
$j=0$. In
particular $W_0$ is degenerate as it should be. The formal difficulty
 is that  condition (A6) of~\cite{SJ_l} is not satisfied: there are two
equivalence classes  $\{0\}$ and $\{1,2,\dots\}$ and the first one is
dominating while the second is not (a ball of type zero is never put in
unless it is drawn). Probably Janson's theory could be
extended to include such cases, but we do not know this.

\item[(b)] 
 Janson found the eigenvalues and eigenvectors of $A$ by
  directly solving the 
  system of linear
  equations. Alternatively, as he himself suggested 
one can proceed 
by induction.
Indeed, 
 after expanding $A_0-\lambda I$ with respect to the first
row and then  repeatedly expanding the resulting cofactors with respect to their last rows 
we see that the characteristic polynomial of $A_0$ is 
\[p_{A_0}(\lambda)=(-1)^{M-1}\lambda(\lambda-2)\prod_{j=0}^{M-1}(\lambda+j).\]

The eigenvalues then may be computed by 
solving the system $(A_0-\lambda I)v=0$.  For example, when $\lambda=2$ this gives  $v_0=0$ and 
\begin{align*}\lambda v_1&=\sum_{i=2}^Mv_i\\
jv_{j-1}&=(\lambda+j)v_{j},\quad 2\le
j<M,\\Mv_{M-1}&=(\lambda-1)v_M,\end{align*}
which is solved by 
\begin{align*}v_{M-1}&=\frac1Mv_M,\\
v_j&=\frac{j+3}{j+1}v_{j+1}=\frac{(j+3)(j+4)\dots(M+1)}{(j+1)(j+2)\dots(M-1)}v_{M-1}=\frac{{M+1\choose2}}{{j+2\choose2}M}v_{M}
,\quad 1\le j\le M-2.
\end{align*}
Note that the fraction on the right--hand side can be written as
\[\frac{M+1}{(j+1)(j+2)}v_M
\]
and choosing $v_M=2/(M+1)$ normalizes $v$ so that $\sum_jv_j=1$. This
is exactly what Janson gives (see \cite[(5.6)]{SJ_s}) except that our
$j$ is his $i+1$. 

The other eigenvectors (left and right) can be computed in a similar fashion. We skip
further details.
\end{itemize}

\bibliographystyle{plain}

 \end{document}